\documentclass[12pt, reqno, a4paper]{amsart}

\usepackage{amsmath}
\usepackage{amssymb}
\usepackage{amsthm}
\usepackage{amscd}
\usepackage{mathrsfs}
\usepackage{mathtools}

\usepackage{enumerate}

\usepackage[margin=1.5in]{geometry}

\newcommand{\ceil}[1]{\left\lceil #1 \right\rceil}
\newcommand{\N}{\mathbb{N}}
\newcommand{\Z}{\mathbb{Z}}
\newcommand{\R}{\mathbb{R}}
\newcommand{\C}{\mathbb{C}}
\newcommand{\E}{\mathbb{E}}

\newcommand{\mb}{\mathbb}

\newcommand{\fp}{\mb{F}_p}
\newcommand{\Mod}[1]{\ (\text{mod}\ #1)}
\newcommand{\ol}{\overline}
\newcommand{\nequiv}{\not\equiv }

\newcommand{\eps}{\epsilon}

\renewcommand{\le}{\leqslant}
\renewcommand{\ge}{\geqslant}

\renewcommand{\hat}{\widehat}
\renewcommand{\tilde}{\widetilde}

\newtheorem{thm}{Theorem}[section]

\newtheorem{prop}[thm]{Proposition}
\newtheorem{lem}[thm]{Lemma}

\theoremstyle{definition}
\newtheorem*{defn}{Definition}

\theoremstyle{claim}

\newtheorem{question}{Question}
\newtheorem*{question*}{Question}

\numberwithin{equation}{section}

\begin{document}
\title{Partition regularity of generalised Fermat equations}
\author{Sofia Lindqvist}
\begin{abstract}
Let $\alpha,\beta,\gamma\in\N$. We prove that given an $r$-colouring of $\fp$ with $p$ prime, there are more than $c_{r,\alpha,\beta,\gamma} p^2$ solutions to the equation $x^\alpha+y^\beta=z^\gamma$ with all of $x,y,z$ of the same colour. Here $c_{r,\alpha,\beta,\gamma}>0$ is some constant depending on the number of colours and the exponents in the equation. This is already a new result for $\alpha=\beta=1$ and $\gamma=2$, that is to say for the equation $x+y=z^2$.
\end{abstract}
\maketitle
\setcounter{tocdepth}{1}
\tableofcontents
\section{Introduction}
Let $p$ be a prime, and write $\fp = \Z/p\Z$. We study partition regularity in $\fp$ of the equation
\begin{equation}
x+y=z^2
\label{problem}
\end{equation}
and of the more general equation
\begin{equation}
x^\alpha+y^\beta = z^\gamma
\label{problem2}
\end{equation}
where $\alpha,\beta,\gamma\in \N$. We are interested in the following two questions, which were asked in \cite{csikvari-gyarmati-sarkozy}. 

\begin{question}\label{q1}
Given an $r$-colouring of $\fp$, will there always be a non-trivial solution to \eqref{problem} with $x,y$ and $z$ the same colour?
\end{question}

\begin{question}\label{q2}
Given an $r$-colouring of $\fp$, will there always be a non-trivial solution to \eqref{problem2} with $x,y$ and $z$ the same colour?
\end{question}

Here any solution with $x=y=z$ is counted as trivial. It turns out that the answer to both questions is indeed positive, provided $p$ is sufficiently large, which is the content of our two main theorems.

\begin{thm}\label{thm:main}
Suppose that $\fp$ is $r$-coloured. Then there are at least $c_rp^2$ monochromatic triples $(x,y,z)$ that satisfy \eqref{problem}, where $c_r>0$ depends on $r$ but not on $p$.
\end{thm}

Note that an immediate corollary of Theorem \ref{thm:main} is that provided $p> \sqrt{2}c_r^{-1/2}$ then there is a monochromatic non-trivial solution to \eqref{problem}, as this ensures more solutions than the trivial ones $x=y=z=0$ and $x=y=z=2$.

\begin{thm}\label{thm:main2}
Suppose that $\fp$ is $r$-coloured. Then there are at least $c_{r,\alpha,\beta,\gamma}p^2$ monochromatic triples $(x,y,z)$ that satisfy \eqref{problem2}, where $c_{r,\alpha,\beta,\gamma}>0$ depends on $r,\alpha,\beta$ and $\gamma$ but not on $p$.
\end{thm}

As noted in \cite{csikvari-gyarmati-sarkozy} there is no general density result for \eqref{problem}. By this it is meant that given a set $A\subset \fp$ with $|A|\ge \alpha p$ for some $\alpha$ which is independent of $p$, there is not necessarily a solution to \eqref{problem} with $x,y,z\in A$. Indeed, the set
\begin{equation*}
A = \left\{ x\in \fp: 0\le x<p/3, 2p/3\le x^2<p \right\}
\end{equation*}
has size $|A|=\frac{1}{9}p + o(p)$, but clearly no solutions to \eqref{problem}. This can be proven using Fourier analysis on $\fp$, which we leave as an exercise for the reader.

It is also worth noting that \eqref{problem} is not partition regular over $\Z/q\Z$, where $q=p^n$ and $p$ is a fixed odd prime greater than $3$. Indeed, by modifying a counterexample given in \cite{csikvari-gyarmati-sarkozy} which shows that \eqref{problem} is not partition regular over $\mb{N}$ by using $16$ colours, one can obtain a counterexample over $\mb{Z}/p^n\mb{Z}$ where the number of colours needed depends on $p$ but not $n$, as shown in \cite{lindqvist_counterex}. This indicates that the primality of $p$ must play an important role in the proof of Theorem \ref{thm:main}, and thus also in Theorem \ref{thm:main2}.

For Question \ref{q2}, the special case of proving partition regularity of the Fermat equation
\[ a^n+b^n=c^n\]
was done in \cite{csikvari-gyarmati-sarkozy}. Their proof uses Schur's lemma \cite{schur} on partition regularity of the equation $x+y=z$. This result strengthened a previous result of Dickson \cite{dickson1} proving existence of non-trivial solutions to this equation over $\fp$.

The proofs of theorems \ref{thm:main} and \ref{thm:main2} use the methods developed by Green and Sanders in \cite{green-sanders}, where they establish partition regularity for quadruples $(x,y,xy,x+y)$ over $\fp$. In several instances the results we need are simply weaker versions of results proved in \cite{green-sanders}, and in these cases our proofs will closely resemble theirs. In some places we will cite the corresponding results in \cite{green-sanders} so that the reader may compare proofs.

Even though the statement of Theorem \ref{thm:main2} is more general than that of Theorem \ref{thm:main}, which is a corollary of the former, the proofs use essentially the same ingredients. For this reason we will introduce all the main tools needed in the context of Theorem \ref{thm:main}, as this avoids the additional clutter caused by the parameters $\alpha,\beta$ and $\gamma$. The proof consists of three main ingredients: a \emph{regularity lemma}, a \emph{counting lemma} and a \emph{Ramsey lemma}. Theorem \ref{thm:main} is proved using these in \S \ref{sec:main}, and then the proofs of the lemmas are given in \S\S \ref{sec:regularity}, \ref{sec:counting}, \ref{sec:ramsey} respectively. Finally, in \S \ref{sec:main2} we mention the necessary modifications of the proof that are needed to establish the more general Theorem \ref{thm:main2}.

\subsection*{Notation}
We write $\mb{E}_{x\in S} = \frac{1}{|S|}\sum_{x\in S}$ for the expectation over some set $S$, as is standard in additive combinatorics. Usually we will have $S=\mb{F}_p$, in which case we will write only $\mb{E}_x$. We write $e_p(x) \coloneqq e^{2\pi i x/p}$. For asymptotic relations we use the usual $O$-notation, where $f=O(g)\Leftrightarrow |f|\le Cg$ for some absolute constant $C$. We will write $O_\alpha(g)$ when the implied constant may depend on the parameter $\alpha$. In addition we use $f\ll g$ as an alternative way of writing $f = O(g)$. If $G$ is a compact Abelian group we denote probability Haar measure on $G$ by $\mu_G$.

\subsection*{Acknowledgements}
The author would like to thank Ben Green for many useful discussions and suggestions. In addition, the author is grateful to the anonymous referees for several useful comments, and in particular for pointing out the fact that partial colourings are not needed in the proof of the main theorem.

This work is supported by Ben Green's ERC Starting Grant 279438, Approximate Algebraic Structure and Applications. 

\section{Counting solutions}\label{sec:count_sol}
For functions $f_1,f_2,f_3: \fp \to \mb{C}$, define 
\begin{equation}
T(f_1,f_2,f_3) = \frac{1}{p^2}\sum_{\substack{x,y,z\in \fp\\x+y=z^2}} f_1(x)f_2(y)f_3(z).
\label{eq:T}
\end{equation}
Clearly $p^2T(1_A,1_A,1_A)$ counts the number of solutions to \eqref{problem} with $x,y,z\in A$, and so we would like to control this quantity. We will show below that $|T(f_1,f_2,f_3)|$ can be controlled by some norm of the functions $f_1,f_2,f_3$. Before stating this result we recall some standard concepts of additive combinatorics.

\subsection*{The Fourier transform}
We will make frequent use of the Fourier transform on $\mb{F}_p$, which is defined by
\begin{equation*}
\hat{f}(\xi) \coloneqq \mb{E}_{x\in\fp} f(x)e_p(-x\xi) = \frac{1}{p}\sum_{x\in\fp} f(x)e_p(-x\xi).
\end{equation*}
Here we have implicitly identified the additive characters $x\mapsto e_p(\xi x)$ with the element $\xi\in\fp$. Some standard properties of the Fourier transform will be useful for us, so we recall them here. The inversion formula is given by
\begin{equation*}
f(x) = \sum_{\xi\in\fp} \hat{f}(\xi) e_p(x\xi).
\end{equation*}
Parseval's identity tells us that
\begin{equation*}
\|f\|_2^2 = \mb{E}_x |f(x)|^2 = \sum_{\xi} |\hat{f}(\xi)|^2.
\end{equation*}
Finally, we define convolutions between two functions $f,g$ as
\begin{equation*}
f*g(x) \coloneqq \mb{E}_y f(x-y)g(y),
\end{equation*}
which is transformed to multiplication under the Fourier transform in light of the identity
\begin{equation*}
\hat{f*g}(\xi) = \hat{f}(\xi)\hat{g}(\xi).
\end{equation*}

\subsection*{The $u_2$- and $u_3$-norms}
We will make frequent use of the $u_2$ and $u_3$ norms, as defined below. 
\begin{defn}
Let $f:\fp\to \mb{C}$. The $u_2$ norm of $f$ is defined by
\begin{equation*}
\|f\|_{u_2} = \sup_{\xi} |\hat{f}(\xi)|.
\end{equation*}
\end{defn}

\begin{defn}
Let $f:\fp\to \mb{C}$. The $u_3$ norm of $f$ is defined by
\begin{equation*}
\|f\|_{u_3} = \sup\Big\{ \big|\mb{E}_{x \in\fp} f(x) e_p(ax^2+bx)\big|:a,b\in \fp \Big\}.
\end{equation*}
\end{defn}

Note that by the definition of $\hat{f}$ it immediately follows that $\|f\|_{u_2} \le \|f\|_{u_3}$.

\subsection*{Norm bounds}
The main content of this section is the following proposition.
\begin{prop}\label{normbound_prop}
If $\|f_1\|_2,\|f_2\|_2 ,\|f_3\|_2 \le 1$ then
\begin{equation*}
|T(f_1,f_2,f_3)| \le\sqrt{2} \min_i \|f_i\|_{u_3}.
\end{equation*}
\begin{proof}
Define $g(w) = \sum_{z^2=w} f_3(z)$, where $g(w)$ is understood to be zero if there is no $z$ such that $z^2=w$. Then by the properties of convolution and the inverse formula for the Fourier transform one has
\begin{equation}
\begin{split}
T(f_1,f_2,f_3) &= \mb{E}_x\mb{E}_y f_1(x)f_2(y)g(x+y)\\
&= \sum_{\xi\in\fp} \hat{f}_1(-\xi)\hat{f}_2(-\xi)\hat{g}(\xi).
\end{split}
\label{eh}
\end{equation}
By the Cauchy--Schwarz inequality and Parseval one gets
\begin{equation*}
\begin{split}
|T(f_1,f_2,f_3)| &\le \|f_1\|_{u_2} \Big(\sum_{\xi}|\hat{f}_2(-\xi)|^2\Big)^{1/2}\Big( \sum_{\xi} |\hat{g}(\xi)|^2 \Big)^{1/2}\\
&= \|f_1\|_{u_2} \|f_2\|_2 \Big( \sum_{\xi} |\hat{g}(\xi)|^2 \Big)^{1/2}.
\end{split}
\end{equation*}
Furthermore,
\begin{equation*}
\begin{split}
\sum_{\xi} |\hat{g}(\xi)|^2 &= \E_x |g(x)|^2
= \E_x \sum_{z^2=x}f_3(z)\sum_{w^2=x}\ol{f_3(w)}\\
&= \E_z |f_3(z)|^2 + \E_z f_3(z)\ol{f_3(-z)}-\frac{|f(0)|^2}{p}
\end{split}
\end{equation*}
where we use Parseval and the definition of $g$. The triangle inequality and Cauchy--Schwarz applied to the second sum above gives that
\begin{equation*}
\Big( \sum_{\xi} |\hat{g}(\xi)|^2 \Big)^{1/2} \le \sqrt{2}\|f_3\|_2,
\end{equation*}
and so
\begin{equation*}
|T(f_1,f_2,f_3)| \le \sqrt{2}\|f_1\|_{u_2}\|f_2\|_2\|f_3\|_2\le \sqrt{2}\|f_1\|_{u_2},
\end{equation*}
by the assumptions on $\|f_2\|_2,\|f_3\|_2$. The exact same argument with the roles of $f_1$ and $f_2$ interchanged gives $|T(f_1,f_2,f_3)|\le\sqrt{2} \|f_2\|_{u_2}$. 

To get the bound in terms of $\|f_3\|_{u_3}$, note that from \eqref{eh}, Cauchy--Schwarz and Parseval one also gets
\begin{equation*}
|T(f_1,f_2,f_3)| \le \|g\|_{u_2} \Big( \sum_{\xi} |f_1|^2 \Big)^{1/2}\Big( \sum_{\xi} |f_2|^2 \Big)^{1/2} = \|g\|_{u_2} \|f_1\|_2 \|f_2\|_2,
\end{equation*}
which is less than $\|g\|_{u_2}$ by the assumptions on $\|f_1\|_2,\|f_2\|_2$. Furthermore, 
\begin{equation*}
\|g\|_{u_2} = \sup_{\xi} \Big| \mb{E}_x g(x)e_p(-x\xi)\Big| = \sup_{\xi} \Big| \mb{E}_z f_3(z) e_p(-z^2\xi)\Big| \le \|f_3\|_{u_3},
\end{equation*}
which completes the proof.
\end{proof}
\end{prop}

In addition to Proposition \ref{normbound_prop} we record the following rather trivial bound.
\begin{lem}\label{l2_bound_lemma}
Let $f_1,f_2,f_3: \fp\to \mb{C}$. Then
\begin{equation}
|T(f_1,f_2,f_3)| \le \|f_1\|_2\|f_2\|_2\|f_3\|_2.
\end{equation}
\begin{proof}
By the triangle inequality and Cauchy--Schwarz we have
\begin{multline*}
|T(f_1,f_2,f_3)| \le \mb{E}_{x,z} |f_1(x)f_2(z^2-x)f_3(z)|\\
\le \mb{E}_z|f_3(z)| (\mb{E}_x |f_1(x)|^2)^{1/2} (\mb{E}_x |f_2(z^2-x)|^2)^{1/2}\\ = \|f_3\|_1 \|f_1\|_2\|f_2\|_2 \le \|f_1\|_2\|f_2\|_2\|f_3\|_2,
\end{multline*}
where we use that as $x$ runs through $\fp$, so does $z^2-x$.
\end{proof}
\end{lem}

Proposition \ref{normbound_prop} and Lemma \ref{l2_bound_lemma} will be used in the proof of Theorem \ref{thm:main} at the end of \S \ref{sec:main}. It is the $u_3$-norm in Proposition \ref{normbound_prop} that decides the correct form of the Regularity lemma, and thus also the other key lemmas in \S \ref{sec:main}.

\section{Quadratic systems and trigonometric polynomials}\label{sec:quad_trig}
Following the approach in \cite{green-sanders}, we make the following definitions. To begin with, write $\mb{G} = (\mb{R}/\mb{Z})\times (\mb{R}/\mb{Z})$. The group operation in $\mb{G}$ will be denoted by $+$. 

\begin{defn}
A quadratic system of dimension $d$ is a map $\Psi:\fp\to \mb{G}^d$ of the form
\begin{equation*}
\Psi(x) = (a_i x^2/p,a_i x/p)_{i=1}^d,
\end{equation*}
where $(a_i)_{i=1}^d\subset \fp^d$.
\end{defn}

If $\Psi$ is a quadratic system we write
\begin{equation}
\Lambda_\Psi \coloneqq \big\{ \xi\in\mb{Z}^d: \xi_1a_1+\cdots+\xi_da_d \equiv 0\Mod p \big\},
\label{lambda}
\end{equation}
where $(a_i)$ are the coefficients appearing in the definition of $\Psi$. In other words, $\Lambda_\Psi$ is a lattice encoding the linear relations between the $a_i$ modulo $p$. Note that $\Lambda_\Psi$ is a lattice of full rank, as $p\mb{Z}^d\subset \Lambda_\Psi$. Further, define the closed subgroup
\begin{equation}
G_\Psi \coloneqq \big\{ g\in (\mb{R}/\mb{Z})^d: \xi\cdot g = 0 \text{ for all } \xi\in\Lambda_\Psi \big\}
\label{eq:Gpsi}
\end{equation}
of $(\mb{R}/\mb{Z})^d$. We also write $H_\Psi = G_\Psi\times G_\Psi$. Note that we now can think of $\Psi$ as a map $\mb{F}_p \to H_\Psi$. 

\begin{lem}\label{duality_lem}
Let $\Lambda_\Psi$ and $G_\Psi$ be defined as in \eqref{lambda} and \eqref{eq:Gpsi} respectively, and suppose that $\lambda\in \mb{Z}^d$ satisfies $\lambda\cdot g=0$ for all $g\in G_\Psi$. Then $\lambda\in \Lambda_\Psi$.
\begin{proof}
Note that $\xi_1a_1+\dots+\xi_da_d \equiv 0 \Mod p$ is the same as saying $\xi\cdot a/p = 0$ in $\mb{R}/\mb{Z}$. It is thus clear that $a/p\in G_\Psi$, and so by assumption $\lambda \cdot a/p = 0$. This is again equivalent to saying $\lambda \cdot a \equiv 0 \Mod p$, and so we have $\lambda\in \Lambda_\Psi$.
\end{proof}
\end{lem}

The above allows us to prove a useful orthogonality relation.
\begin{lem}\label{orthog_lemma}
It holds that
\begin{equation}
\int e(\xi\cdot t) \: d\mu_{G_\Psi}(t) = 1_{\Lambda_\Psi} (\xi).
\end{equation}
\begin{proof}
If $\xi\in\Lambda_\Psi$ it is clear that the above integral is $1$. Suppose instead that the integral is nonzero. Take $g\in G_\Psi$ and make the substitution $t\mapsto t+g$, which preserves Haar measure, to get
\begin{equation*}
\int e(\xi\cdot t)\: d\mu_{G_\Psi} (t) = e(\xi\cdot g) \int e(\xi\cdot t)\:d\mu_{G_\Psi}(t).
\end{equation*}
This gives $\xi\cdot g=0$ in $\mb{R}/\mb{Z}$, and as this holds for any $g\in G_\Psi$ it follows from Lemma \ref{duality_lem} that $\xi\in \Lambda_\Psi$.
\end{proof}
\end{lem}

In addition to the notion of quadratic systems we will need the notion of trigonometric polynomials, and the trig-norm, as defined next.

\begin{defn}
Let $F:\mb{G}^d\to \mb{C}$ be a function. We write $\|F\|_\mathrm{trig}$ for the smallest $M$ such that $F$ has a Fourier expansion
\begin{equation}
F(\theta_1,\theta_2) = \sum_{\|\xi_1\|_1, \|\xi_2\|_1\le M} \hat{F}(\xi_1,\xi_2) e(\xi_1\cdot \theta_1+\xi_2\cdot\theta_2)
\label{trig}
\end{equation}
and $\sum_{\xi_1,\xi_2}|\hat{F}(\xi_1,\xi_2)|\le M$. If $\|F\|_\mathrm{trig}<\infty$ we say that $F$ is a \emph{trigonometric polynomial}.
\end{defn}
The Fourier coefficients appearing in the above definition are given by
\begin{equation}
\hat{F}(\xi_1,\xi_2) = \int F(\theta_1,\theta_2) e(-\xi_1\cdot\theta_1-\xi_2\cdot\theta_2)\: d\theta_1d\theta_2.
\label{F_fourier}
\end{equation}
We also note that the trigonometric polynomials are dense in $C(\mb{G}^d)$, the space of continuous, complex valued functions on $\mb{G}^d$.

The following proposition can be thought of as the most basic form of the counting lemma presented in \S \ref{sec:counting}, and will be used to prove Lemma \ref{H_approx_psi} on approximating points in $H_\Psi$ by points in the image of $\Psi$.

\begin{lem}\label{lem:baby_count}
Let $\Psi$ be a quadratic system of dimension $d$, and let $F:\mb{G}^d\to \mb{C}$ be a trigonometric polynomial. Then
\begin{equation*}
\mb{E}_x F\circ \Psi(x) = \int F\:d\mu_{H_\Psi} + O(\|F\|_{\mathrm{trig}} p^{-1/2}).
\end{equation*}
\end{lem}
We give the proof of Lemma \ref{lem:baby_count} in \S \ref{sec:counting}, as the proof is just a much easier version of the proof of the counting lemma.

At this point we need to establish a notion of absolute value of points in $\mb{G}^d$. For a point $(\theta,\phi)\in \mb{G}$ we write $|(\theta,\phi)| \coloneqq \max \{ \|\theta\|_{\mb{R}/\mb{Z}}, \|\phi\|_{\mb{R}/\mb{Z}} \}$, where $\| x \|_{\mb{R}/\mb{Z}}$ is the distance from $x$ to the nearest integer. We then let $|\eta| = \max_{i=1}^d\{|\eta_i|\}$ for $\eta\in \mb{G}^d$. 

Before proving our final result of the section we will need the following result from \cite{taovu2010}.
\begin{lem}\label{lem:taovu}
For $\epsilon\ge 0$ it holds that
\begin{equation*}
\mu_{H_\Psi}\big( \{ x\in H_\Psi : |x|\le \epsilon\} \big) \ge \epsilon^{2d}.
\end{equation*}
\end{lem}
\begin{proof}
This follows directly from \cite[Lemma 4.20]{taovu2010}. Indeed, in the notation of the lemma we take the ambient group $Z$ to be $H_\Psi$ and set $S=\{\xi_1,\dots,\xi_d,\zeta_1,\dots,\zeta_d\}$ where $\xi_i\cdot (\theta,\phi) = \theta_i$ and $\zeta_i\cdot(\theta,\phi)=\phi_i$.
\end{proof}

The next lemma tells us that any point in $H_\Psi$ can be well approximated by $\Psi(x)$ for some value of $x$, provided $p$ is large, and so in some sense we are saying that $H_\Psi$ is well approximated by the image of $\Psi$. This result will be needed in the proof of Proposition \ref{combination_prop}.
\begin{lem}[{Cf. \cite[Corollary 3.4]{green-sanders}}]\label{H_approx_psi}
There is a function $p_1:\Z_{\ge 0}\times (0,1]\to \R_{\ge 0}$ such that the following holds. For any $d$-dimensional quadratic system $\Psi$ and $h \in H_\Psi$ we have
\begin{equation*}
\# \{x\in \mb{F}_p : |\Psi(x)-h|\le \epsilon \} \ge \frac{1}{8}\left( \frac{\epsilon}{2} \right)^{2d}p
\end{equation*}
provided $p\ge p_1(d,\epsilon)$.
\end{lem}
\begin{proof}
Fix $h\in H_\Psi$. We can find a trigonometric polynomial $F$ satisfying $-\delta \le F \le 2$ on $\mb{G}^d$, $F(\theta)\le 0$ for $|\theta-h|> \epsilon$ and $F(\theta)\ge 1$ for $|\theta-h|\le \epsilon/2$. Furthermore $\|F\|_\mathrm{trig}$ is bounded in terms of $\epsilon,\delta$ and $d$. Applying Lemma \ref{lem:taovu} and using translation invariance of Haar measure we have
\begin{equation*}
\int F\: d\mu_{H_\Psi} = \int F(\theta-h)\: d\mu_{H_\Psi}(\theta) \ge \left(\frac{\epsilon}{2}\right)^{2d} - \delta.
\end{equation*}
We set $\delta = \frac{1}{2}\left( \frac{\epsilon}{2} \right)^{2d}$ and note that $\mb{E}_x F\circ \Psi (x) \le 2 \mu_{\mb{F}_p} (\{x: |\Psi(x)-h|\le \epsilon\})$, as $F\le 2$ everywhere. By invoking Lemma \ref{lem:baby_count} and choosing $p$ larger than some function $p_1(d,\epsilon)$ we then have
\begin{equation*}
\mu_{\mb{F}_p}(\{x: |\Psi(x)-h|\le \epsilon\}) \ge \frac{1}{2}\left( \left( \frac{\epsilon}{2}\right)^{2d} - \frac{1}{2}\left( \frac{\epsilon}{2} \right)^{2d} \right) - \frac{1}{8}\left( \frac{\epsilon}{2} \right)^{2d} = \frac{1}{8}\left( \frac{\epsilon}{2} \right)^{2d}.
\end{equation*}
\end{proof}

\section{Proof of main theorem for the case $x+y=z^2$}\label{sec:main}
In order to prove Theorem \ref{thm:main} we will make use of the following three key lemmas.

\begin{lem}[Regularity lemma]\label{lem:regularity_lemma}
There are functions $p_3,M:\mb{Z}_{\ge 0}\times (0,1]\to \mb{R}_{\ge 0}$ such that the following holds. Let $c:\fp \to [r]$ be an $r$-colouring of $\fp$ and let $\eps>0$. Then there is a quadratic system $\Psi$ of dimension $d$, functions $F_1,\dots,F_r: \mb{G}^d\to \mb{R}_{\ge 0}$, and functions $g_1,\dots,g_r:\fp \to [-1,1]$, such that the following holds provided $p\ge p_3(r,\epsilon)$.
\begin{enumerate}[(i)]
\item $\|F_i\|_\mathrm{trig} \le M(r,\epsilon)$ for all $i$;\\
\item $d\le 8 r \epsilon^{-2}+1$;\\
\item $\|F_i\circ \Psi-g_i\|_2\le \epsilon$ for all $i$;\\
\item $\|1_{c^{-1}(i)}-g_i\|_{u_3} \le \epsilon$ for all $i$;\\
\item $\sum_i F_i\circ \Psi\ge 1$ pointwise.\\
\end{enumerate}
\end{lem}

\begin{lem}[Counting lemma]\label{counting_lemma}
Let $\Psi$ be a $d$-dimensional quadratic system and $F:\mb{G}^d\to \mb{C}$ be a trigonometric polynomial. Then
\begin{multline}\label{counting_eq}
T(F\circ \Psi, F\circ \Psi, F\circ \Psi ) = \int F(t,u)F(t',u')F(u+u',u^{\prime\prime})\:d\mu_{G_\Psi}^{\otimes 5}(u,u',u^{\prime\prime},t,t')\\+ O\left(p^{-1/2}M^3 \right),
\end{multline}
where $M=\|F\|_\mathrm{trig}$.
\end{lem}

\begin{lem}[Ramsey lemma]\label{ramsey_lemma}
There is a positive function $\rho:\N \to (0,1]$ with the following property. Suppose that $G$ is a compact Abelian group with Haar probability measure $\mu$. Assume also that $F_1,\dots,F_r: G\times G\to \mb{R}_{\ge 0}$ are continuous functions that satisfy $\sum_{i=1}^r F_i(g_1,g_2)\ge 1$ for all $(g_1,g_2)\in G\times G$. Then there is an $i\in [r]$ such that
\begin{equation}
\int F_i(t,u)F_i(t',u')F_i(u+u',u^{\prime\prime})\:d\mu^{\otimes 5}(u,u',u^{\prime\prime},t,t') \ge \rho(r).
\label{rho}
\end{equation}
\end{lem}

The Regularity lemma allows us to replace the characteristic functions of the colour classes by structured functions of the form $F\circ \Psi$, up to a small error. For these special functions the counting lemma tells us that counting configurations of the form $x+y=z^2$ in $\fp$ is the same as counting some linear configuration in $H_\Psi$. Finally, the Ramsey lemma establishes that there are in fact sufficiently many monochromatic linear configurations in $H_\Psi$.

Lemmas \ref{counting_lemma} and \ref{ramsey_lemma} will be used together, so we therefore state the combination as a separate proposition.
\begin{prop}[{Cf. \cite[Proposition 4.5]{green-sanders}}]\label{combination_prop}
There is a function $p_2:\mb{R}_{\ge 0}\times \mb{Z}_{\ge 0}\times \mb{Z}_{\ge 0}\to \mb{R}_{\ge 0}$ such that the following holds. Let $\Psi$ be a quadratic system of dimension $d$. Suppose that $F_1,\dots,F_r: \mb{G}^d\to \R_{\ge 0}$ satisfy $\|F_i\|_\mathrm{trig}\le M$ and $\sum_i F_i\circ \Psi\ge 1$ pointwise on $\fp$. Then there is some $i$ such that
\begin{equation}
T(F_i\circ\Psi, F_i\circ \Psi,F_i\circ\Psi) \ge 2^{-4} \rho(r),
\end{equation}
with $\rho$ as given by Lemma \ref{ramsey_lemma}, provided $p\ge p_2(M,r,d)$.
\begin{proof}
Apply Lemma \ref{counting_lemma} to each $F_i$ composed with $\Psi$ to get that
\begin{multline*}
T(F_i\circ\Psi,F_i\circ\Psi,F_i\circ\Psi) \\
\ge \int F_i(t,u)F_i(t',u')F_i(u+u',u^{\prime\prime})\:d\mu_{G_\Psi}^{\otimes 5} (t,t',u,u',u^{\prime\prime}) - \frac{1}{16}\rho(r),
\end{multline*}
provided $p$ is large enough, say $p\ge p_2'(M,r)$.

Let $h\in H_\Psi$. Note that $|F_i(w)-F_i(v)|\le 4\pi d M^2|w-v|$ for all $v,w\in\mb{G}^d$, and so
\begin{equation*}
\Big|\sum_{i=1}^r F_i(h)-\sum_{i=1}^r F_i(\Psi(z))\Big| \le 4\pi dr M^2|\Psi(z)-h|
\end{equation*}
for any $z\in \fp$. By Lemma \ref{H_approx_psi} we can find a $z$ such that $|\Psi(z)-h|\le \frac{1}{8\pi d r M^2}$ provided $p\ge p_1(d,\frac{1}{8\pi d r M^2})$. For this value of $z$ we then have
\begin{equation}
\sum_{i=1}^r F_i(h) \ge \sum_{i=1}^r F_i\circ\Psi(z) -\frac{1}{2} \ge \frac{1}{2}.
\end{equation}
Applying Lemma \ref{ramsey_lemma} to the functions $(2F_i)_{i=1}^r$ then gives that there is some $i$ such that
\begin{equation}
T(F_i\circ \Psi,F_i\circ\Psi,F_i\circ\Psi) \ge 2^{-3}\rho(r)-2^{-4}\rho(r) = 2^{-4}\rho(r),
\end{equation}
provided $p\ge p_2(M,r,d)$, where we define $p_2(M,r,d)$ to be the largest of $p_2'(M,r)$ and $p_1(d,\frac{1}{8\pi d r M^2})$.
\end{proof}
\end{prop}

We are now ready to prove the main theorem. Note that in contrast to \cite[proof of Proposition 4.1 p. 22]{green-sanders} we will not need to induct on the number of colours.

\begin{proof}[Proof of Theorem \ref{thm:main}]
Given the number of colours $r$, fix $\epsilon = 2^{-7}3^{-1}\rho(r)$ and
\begin{equation}
p_0(r) = \max\Big\{ p_3(r,\epsilon),\sup_{d\le 8r\epsilon^{-2}+1} p_2\big(M(r,\epsilon),r,d\big) \Big\},
\end{equation}
where $\rho$ and $p_2$ are the functions appearing in Proposition \ref{combination_prop} and $p_3$ and $M$ are the functions appearing in Lemma \ref{lem:regularity_lemma}. Note that the above supremum over $d$ runs over positive integers, and so it is finite. Assume first that $p\ge p_0$. Apply Lemma \ref{lem:regularity_lemma} with the above $\eps$ and our colouring $c$ to find $F_i,g_i,\Psi$ and $d$ that satisfy the properties listed in the Lemma, provided $p\ge p_3(r,\epsilon)$, which is satisfied by the choice of $p_0$. By Proposition \ref{combination_prop} there is some $i\in [r]$ such that
\begin{equation*}
T(F_i\circ \Psi,F_i\circ\Psi, F_i\circ \Psi) \ge 2^{-4}\rho(r)
\end{equation*}
provided $p\ge p_2(M(r,\epsilon),r,d)$. This holds by the choice of $p_0$ and property $(ii)$ from Lemma \ref{lem:regularity_lemma}. Note also that $\|g_i\|_2\le 1$ as $g_i: \fp\to [-1,1]$, $\|F_i\circ\Psi-g_i\|_2\le \eps \le 1$ and so $\|F_i\circ\Psi\|_2\le 2$. Lemma \ref{l2_bound_lemma} then gives
\begin{multline*}
|T(g_i,g_i,g_i)-T(F_i\circ\Psi,F_i\circ\Psi,F_i\circ\Psi)| \le \\
|T(g_i-F_i\circ\Psi,g_i,g_i)| + |T(F_i\circ\Psi,g_i-F_i\circ\Psi, g_i)| + |T(F_i\circ\Psi, F_i\circ\Psi, g_i-F_i\circ\Psi)|\\
\le 7 \epsilon,
\end{multline*}
as $\|F_i\circ\Psi-g_i\|_2 \le \epsilon$, so that 
\begin{equation}
|T(g_i,g_i,g_i)|\ge 2^{-4} \rho(r) - 7\eps.
\label{gi_bound}
\end{equation}

Let
\[A_i = c^{-1}(i),\quad i=1,\dots, r.\]
By Proposition \ref{normbound_prop} we get that
\begin{multline*}
|T(g_i,g_i,g_i)-T(1_{A_i},1_{A_i},1_{A_i})| \le\\
|T(g_i-1_{A_i},g_i,g_i)|+|T(1_{A_i},g_i-1_{A_i},g_i)| + |T(1_{A_i},1_{A_i},g_i-1_{A_i})| \\
\le 3\sqrt{2} \|g_i-1_{A_i}\|_{u_3} \le 3\sqrt{2}\epsilon,
\end{multline*}
where the final inequality follows by property $(iv)$ of Lemma \ref{lem:regularity_lemma}. Now this combined with \eqref{gi_bound} gives
\begin{equation*}
T(1_{A_i},1_{A_i},1_{A_i}) \ge 2^{-4} \rho(r)-7\eps-3\sqrt{2}\epsilon > 2^{-5} \rho(r)
\end{equation*}
by our choice of $\eps$.

If instead $p\le p_0$ we have the two trivial solutions $x=y=z=0$ and $x=y=z=2$, giving at least $\frac{2}{p_0(r)^2}p^2$ monochromatic solutions to \eqref{problem}. Finally we set
\[ c_r = \min\left\{ \frac{2}{p_0(r)^2}, 2^{-5} \rho(r)\right\}, \]
which finishes the proof.

\end{proof}

\section{Regularity lemma}\label{sec:regularity}
The conclusion of Lemma \ref{lem:regularity_lemma} is implied by the regularity lemma given in \cite{green-sanders}. The proof given here is therefore also just the necessary parts of the proof given by Green--Sanders in their paper. Nevertheless we believe it is instructive to see the full proof written out in this simpler setting, as many of the unpleasant technicalities of \cite{green-sanders} go away in this case.

We begin by defining intervals on $\mb{G}^d$. Let $I(x,R) = \big[(2x-1)/2R,(2x+1)/2R\big)$ be an interval in $\mb{R}/\mb{Z}$ and define
\begin{equation*}
I_{R;t,u} = \big\{ (\theta,\phi)\subset \mb{G}^d:\theta_j \in I(t_j,R), \phi_j\in I(u_j,R) \text{ for }j=1,\dots,d\big\}.
\end{equation*}
Next, given a quadratic system $\Psi$, consider the $\sigma$-algebra generated by $\Psi^{-1}(I_{R;t,u})$ for $t,u\in \{0,\dots,R-1\}^d$, and let $\Pi_R^{\Psi}$ be the projection operator onto this $\sigma$-algebra. Explicitly, for any $f:\fp\to\C$, we have that
\begin{equation*}
\Pi_R^\Psi f (x) = \frac{1}{|A(x)|} \sum_{x'\in A(x)} f(x'),
\end{equation*}
where $A(x) = \Psi^{-1}(I_{R;t,u})$ with $t,u$ such that $\Psi(x)\in I_{R;t,u}$.

\begin{lem}\label{u3bad_projbad_lemma}
Suppose that $\|f\|_{\infty} \le 1$, $\|f\|_{u_3}\ge \delta$ and $R>8\pi \delta^{-1}$. Then there is a function $g\in L^\infty(\fp)$ with $\|g\|_\infty \le 1$ and a quadratic system $\Phi$ of dimension $2$, such that $|\langle f,\Pi_R^\Phi g\rangle| \ge \frac{1}{2}\delta$.
\begin{proof}
By the definition of the $u_3$-norm there are $a_1,a_2$ such that
\begin{equation}
\big|\mb{E}_x f(x) \overline{e_p(a_1x^2+a_2x)}\big| \ge \delta.
\label{eq:u3_large}
\end{equation}
Let $\Phi(x) = \left(\frac{a_ix^2}{p},\frac{a_ix}{p}\right)_{i=1,2}$ and $F(\theta_1,\phi_1,\theta_2,\phi_2) = e(\theta_1+\phi_2)$, and set $g=F\circ \Phi$. Now \eqref{eq:u3_large} states precisely that $|\langle f, g\rangle |\ge \delta$, and $\|g\|_\infty \le 1$ follows from $\|F\|_{\infty} \le 1$. Furthermore, 
\begin{equation*}
\|\Pi_R^\Phi g-g\|_\infty \le \sup_x \sup_{x' \in A(x)} |F(\Phi(x'))-F(\Phi(x))| \le 4\pi R^{-1},
\end{equation*}
as $|F(x)-F(y)| \le 4\pi |x-y|$. The assumption on $R$ gives that this is less than $\frac{1}{2}\delta $. This together with $\|f\|_\infty\le 1$ gives
\begin{equation*}
|\langle f, \Pi_R^\Phi g\rangle|\ge |\langle f, g\rangle| - |\langle f,g-\Pi_R^\Phi g\rangle |\ge \frac{1}{2}\delta,
\end{equation*}
as required.
\end{proof}
\end{lem}

By repeated application of the above lemma and an energy increment argument we are able to establish the following Koopman von Neumann type Lemma.
\begin{lem}\label{koopman_von_neumann}
Suppose that $f_1,\dots,f_r:\fp \to \mb{C}$ are such that $\|f_i\|_\infty \le 1$ for all $i\in \{1,\dots,r\}$ and that $R>16 \pi \delta^{-1}$. Then there is a quadratic system $\Psi$ with $\dim \Psi \le 8r\delta^{-2}+1$ such that $\|f_i-\Pi_R^{\Psi}f_i\|_{u_3}\le \delta$ for all $i\in\{1,\dots,r\}$.

\end{lem}
\begin{proof}
We construct $\Psi$ in an iterative manner. To begin with, pick any $1$-dimensional quadratic system and let this be $\Psi_0$. At stage $j$, define $f_{i,j}=f_i-\Pi_R^{\Psi_j} f_i$. If $\|f_{i,j}\|_{u_3}\le \delta$ for all $i$ we are done and take $\Psi=\Psi_j$ of dimension $2j+1$. If not, there is some $i$ such that $\|f_{i,j}\|_{u_3}>\delta$, in which case we can apply Lemma \ref{u3bad_projbad_lemma} to $\frac{1}{2}f_{i,j}$. The factor $\frac{1}{2}$ is added as $\|f_{i,j}\|_\infty \le 2\|f_i\|_\infty \le 2$. This gives a quadratic system $\Phi$ of dimension $2$ and a function $g$ with $\|g\|_\infty\le 1$, such that
\begin{equation*}
|\langle f_{i,j}, \Pi_R^\Phi g\rangle| \ge \frac{1}{2}\delta.
\end{equation*}
Assume that $\Psi_j$ is defined by the coefficients $(a_i)_{i=1}^d$ and that $\Phi$  is defined by the coefficients $(b_i)_{i=1}^2$. We then define $\Psi_{j+1}$ as the quadratic system of dimension $d+2$ with coefficients $(a_1,\dots,a_d,b_1,b_2)$. Noting that $\Pi_R^\Phi$ is idempotent and self adjoint, and that $\Pi_R^\Phi\Pi_R^{\Psi_{j+1}}=\Pi_R^\Phi$ we get that
\begin{equation*}
\langle f_{i,j+1},\Pi_R^\Phi g\rangle = \langle \Pi_R^\Phi f_{i,j+1},\Pi_R^\Phi g\rangle = \langle \Pi_R^\Phi f_i - \Pi_R^{\Phi}\Pi_R^{\Psi_{j+1}} f_i, \Pi_R^\Phi g\rangle  = 0.
\end{equation*}
This in turn gives
\begin{equation*}
|\langle \Pi_R^{\Psi_{j+1}} f_i - \Pi_R^{\Psi_j} f_i, \Pi_R^{\Phi} g\rangle| =|\langle f_{i,j},\Pi_R^\Phi g\rangle - \langle f_{i,j+1}, \Pi_R^\Phi g\rangle| \ge \frac{1}{2}\delta,
\end{equation*}
and by the Cauchy--Schwarz inequality and the fact that $\|g\|_2\le \|g\|_\infty \le 1$ we get that
\begin{equation*}
\|\Pi_R^{\Psi_{j+1}}f_i-\Pi_R^{\Psi_j}f_i\|_2 \ge \frac{1}{2}\delta.
\end{equation*}
Using the fact that $\Pi_R^{\Psi_j}$ is idempotent and self-adjoint, and that $\Pi_R^{\Psi_j}\Pi_R^{\Psi_{j+1}} = \Pi_R^{\Psi_j}$, we get $\langle \Pi_R^{\Psi_j}f_i,\Pi_R^{\Psi_{j+1}} f_i \rangle = \|\Pi_R^{\Psi_j}f_i\|_2^2$, which then gives
\begin{equation*}
\|\Pi_R^{\Psi_{j+1}}f_i\|_2^2- \|\Pi_R^{\Psi_j} f_i\|_2^2 = \|\Pi_R^{\Psi_{j+1}}f_i-\Pi_R^{\Psi_j} f_i\|_2^2 \ge \frac{1}{4}\delta^2. 
\end{equation*}
This is what is needed to complete an energy increment argument. Indeed, defining
\begin{equation*}
E_j = \sum_{i=1}^r \|\Pi_R^{\Psi_j}f_i\|_2^2,
\end{equation*}
we get that
\begin{equation*}
E_{j+1}-E_j \ge \frac{1}{4}\delta^2.
\end{equation*}
The trivial bound $E_j \le r$ gives that we can continue the iteration process at most $4r\delta^{-2}$ times, and we are done.
\end{proof}

Note that by definition of $\Pi_R^\Psi$ we can write $\Pi_R^\Psi f = F_0\circ \Psi$ for some $F_0$. The only thing that remains to settle before embarking on the proof of Lemma \ref{lem:regularity_lemma} is a way of controlling the trig norm of $F_0$, which is what the following lemma does.

\begin{lem}[{Cf. \cite[Lemma 5.3]{green-sanders}}]\label{trig_lem}
There are functions $M_0,p_4:(0,1]\times \N \times \mb{N}\to \mb{R}_{\ge 0}$ such that the following holds. Fix $R\in\N$, a quadratic system $\Psi$ of dimension $d\in \N$ and $\eps\in(0,1]$. Then for all functions $f:\fp\to [0,1]$, provided $p\ge p_4(\epsilon,d,R)$, there is a function $F:\mb{G}^d\to \mb{R}_{\ge 0}$ such that
\begin{enumerate}[(i)]
\item $F\circ \Psi\ge \Pi_R^\Psi f$ pointwise;\\
\item $\|F\|_\mathrm{trig} \le M_0(\epsilon,d,R)$;\\
\item $\|F\circ \Psi-\Pi_R^\Psi f\|_2\le \epsilon$.
\end{enumerate}
\begin{proof}
We start by defining $A_{R;t,u} = \Psi^{-1}(I_{R;t,u})$, and note that $\Pi_R^{\Psi} f$ is constant on each of $A_{R;t,u}$. Denote this constant value by $c_{R;t,u}$, and note that then $0\le c_{R;t,u}\le 1$ by assumption. We will construct $F$ by approximating $1_{I_{R;t,u}}$ for each $t,u\in\{0,\dots,R-1\}$ and then adding these with weights $c_{R;t,u}$. 

Let $\eta>0$ and define
\begin{equation*}
I_{R;t,u}^{\pm} = \Big\{(\theta,\phi)\in\mb{G}^d : \Big\|\theta_j-\frac{t_j}{R}\Big\|_{\mb{R}/\mb{Z}}< \frac{1}{2R}\pm \eta, \Big\|\phi_j - \frac{u_j}{R}\Big\|_{\mb{R}/\mb{Z}}< \frac{1}{2R}\pm \eta\Big\}.
\end{equation*}
As the trigonometric polynomials are dense in $C(\mb{G}^d)$, given $\delta>0$ which we will specify later, we can find $F_{R;t,u}:\mb{G}^d\to\mb{R}$ which satisfy the following properties:
\begin{enumerate}
\item $0\le F_{R;t,u}\le 1+\delta$;\\
\item $F_{R;t,u}(z)\ge 1$ for all $z\in I_{R;t,u}$;\\
\item $F_{R;t,u}(z)\le \delta$ for all $z\not\in I_{R;t,u}^+$;\\
\item $\|F_{R;t,u}\|_\mathrm{trig} \le M_1(\delta,\eta,d,R)$
\end{enumerate}
for some function $M_1 : (0,1) \times (0,1) \times \N \times \N \to \R_{\ge 0}$. With this in hand we set
\begin{equation*}
F = \sum_{t,u\in \{0,\dots,R-1\}^d} c_{R;t,u} F_{R;t,u},
\end{equation*}
and note that $\Pi_R^\Psi f=F_0\circ \Psi$, where
\begin{equation*}
F_0 = \sum_{t,u\in\{0,\dots,R-1\}^d} c_{R;t,u} 1_{I_{R;t,u}}.
\end{equation*}
It remains to verify that $F$ indeed satisfies each of the properties stated in the lemma, with appropriate choices of $\eta=\eta(\epsilon,d,R)$ and $\delta=\delta(\epsilon,d,R)$.

\begin{enumerate}[(i)]
\item If $z\in A_{R;t,u}$ then $\Pi_R^\Psi f(z) = c_{R;t,u}$, and by property $(1)$ and $(2)$ above we get that $F(z)\ge c_{R;t,u} F_{R;t,u}(z)\ge c_{R;t,u}$, and so $(i)$ holds. \\
\item Set $M_0(\epsilon,d,R) = R^{2d}M_1(\delta(\epsilon,d,R), \eta(\epsilon,d,R), d ,R)$ and then the result follows by property $(4)$ above.\\
\item Begin by noting that $I_{R;t,u}^{-}\cap I_{R;t',u'}^{+}=\emptyset$ unless $(t,u)=(t',u')$. Define
\begin{equation*}
E = \bigcup_{t,u} I^+_{R;t,u}\setminus I^{-}_{R;t,u}
\end{equation*}
and let $z$ be such that $z\not\in E$. As $\bigcup_{t,u} I_{R;t,u}^+$ covers $\mb{G}^d$ we must then have that $z\in I^-_{R;t,u}$ for some $t,u$. Then
\begin{equation}\label{I-bd}
|F(z)-F_0(z)| \le \sum_{(t' , u')\ne (t,u)} c_{R;t',u'}F_{R;t',u'}(z) + c_{R;t,u}\delta \le R^{2d}\delta
\end{equation}
by properties $(1)$, $(2)$ and $(3)$ above.

Next, we'll need to bound the size of $\Psi^{-1}(E)$. If $J$ is an interval in $\mb{R}/\mb{Z}$ not containing $0$ we have that
\begin{equation}\label{interval_bd}
\#\left\{ x\in \fp : ax/p \in J \right\} \le p|J| + C,
\end{equation}
for any $a\in \fp$ and an absolute constant $C$. This is because $ax$ takes on each value in $\fp$ at most once (exactly once if $a\ne 0$). We now note that
\begin{equation*}
E \subset \bigcup_{u_1\in \{0,\dots,R-1\}} \Big\{ (\theta,\phi)\in \mb{G}^d: \Big\| \phi_1- \frac{2u_1+1}{2R}\Big\|_{\mb{R}/\mb{Z}} < \eta \Big\}
\end{equation*}
and by \eqref{interval_bd} we then get $|\Psi^{-1}(E)|\le  R(2\eta p + C) < 3R\eta p $, provided $p\ge C\eta^{-1}$. To ensure this, define $p_4(\epsilon,d,R) = C\eta^{-1}$. Using property $(1)$ and $|f(x)|\le 1$ then gives
\begin{equation}\label{E_bd}
\sum_{x \in  \Psi^{-1}(E)} \big|F\circ \Psi(x) - F_0\circ \Psi(x)\big|^2 < (R^{2d}(1+\delta)+1)^2 3 R\eta p.
\end{equation}

Combining \eqref{I-bd} and \eqref{E_bd} now gives
\begin{equation}\label{l2_bound}
  \begin{split}
    \|F\circ \Psi - \Pi_R^\Psi f\|_2^2 = \frac{1}{p} \Big(\sum_{x\not \in \Psi^{-1}(E)}+\sum_{x\in \Psi^{-1}(E)} \Big)(F\circ\Psi(x)-F_0\circ \Psi(x))^2\\
    <  R^{4d} \delta^2 + 27 R^{4d+1}\eta =  \epsilon^2,
  \end{split}
\end{equation}
where we have defined
\begin{equation*}
  \eta = \frac{\epsilon^2}{28 R^{4d+1}}, \quad\quad \delta = \frac{\epsilon}{\sqrt{28} R^{2d}}.
\end{equation*}
\end{enumerate}
\end{proof}
\end{lem}

At this point we can complete the proof of Lemma \ref{lem:regularity_lemma}. 
\begin{proof}[Proof of Lemma \ref{lem:regularity_lemma}]
Assume that a colouring $c$ and $\eps>0$ is given. Write $f_i=1_{c^{-1}(i)}$ for $i=1,\dots,r$ and apply Lemma \ref{koopman_von_neumann} to the $f_i$ with $\delta=\epsilon$ and $R=\ceil{16\pi \eps^{-1}}+1$ to get a quadratic system $\Psi$ of dimension $d\le 8r\epsilon^{-2}+1$. Now we apply Lemma \ref{trig_lem} to each of the $f_i$ with this choice of $\Psi$ and $R$ as before to get functions $F_i$ satisfying the conclusion of the lemma. Define $g_i = \Pi_R^\Psi f_i$ for $i=1,\dots,r$. We will now show that the functions $F_i,g_i$ satisfy the desired properties.
\begin{enumerate}[$(i)$]
\item As $d\le 8r\epsilon^{-2}+1$ we set $M(r,\epsilon) = \sup_{d\le 8r \epsilon^{-2}+1} M_0(\epsilon, d, R)$, where $M_0$ is as given in Lemma \ref{trig_lem}. Provided that $p$ satisfies the conditions of Lemma \ref{trig_lem}, this follows immediately from property $(ii)$ of Lemma \ref{trig_lem}. We therefore also set $p_3(r,\epsilon) = \sup_{d\le 8r \epsilon^{-2}+1} p_4(\epsilon, d, R)$, with $p_4$ as given in Lemma \ref{trig_lem}. Note that both of these supremums are taken over integer values of $d$, and are therefore finite.
\item This follows immediately by the choice of $\Psi$ from Lemma \ref{koopman_von_neumann}.
\item This follows immediately by the choice of $F_i$ from Lemma \ref{trig_lem}.
\item By the choice of $\Psi$ from Lemma \ref{koopman_von_neumann} together with the definition of $g_i$, we get
\begin{equation*}
    \|1_{c^{-1}(i)}-g_i\|_{u_3} = \|f_i-\Pi_R^\Psi f_i\|_{u_3} \le \epsilon.
\end{equation*}
\item By property $(i)$ of Lemma \ref{trig_lem} we get that
\begin{equation*}
\sum_i F_i\circ \Psi \ge \sum_i \Pi_R^\Psi f_i = \Pi_R^\Psi \sum_i f_i = \Pi_R^\Psi 1 = 1.
\end{equation*}
\end{enumerate}
\end{proof}

\section{Counting lemma}\label{sec:counting}
Before proving the main counting lemma, we give the proof of Lemma \ref{lem:baby_count}.
\begin{proof}[Proof of Lemma \ref{lem:baby_count}]
Write $F$ in terms of its Fourier coefficients to get
\begin{equation*}
\mb{E}_x F\circ\Psi (x) = \sum_{\xi_1,\xi_2} \hat{F}(\xi_1,\xi_2) \mb{E}_x e_p(\xi_1\cdot a\: x^2 + \xi_2\cdot a\: x)
\end{equation*}
We now use that
\begin{equation}
\mb{E}_x e_p(ax^2+bx) \ll p^{-1/2} \quad\quad\text{if } (a,b)\ne(0,0),
\label{quadratic sum}
\end{equation}
in order to discard all terms for which $\xi_1\cdot a \ne 0$ or $\xi_2\cdot a \ne 0$. This contributes an error $O(Mp^{-1/2})$, where $M=\|F\|_\mathrm{trig}$, and so we have
\begin{equation*}
\mb{E}_x F\circ\Psi (x) = \sum_{\xi_1,\xi_2\in\Lambda_\Psi} \hat{F}(\xi_1,\xi_2) + O(Mp^{-1/2}).
\end{equation*}
At the same time we have
\begin{equation*}
\int F \: d\mu_{H_\Psi} = \sum_{\xi_1,\xi_2} \hat{F}(\xi_1,\xi_2) \int e(\xi_1\cdot t)e(\xi_2\cdot u) \: d\mu_{H_\Psi}(t,u),
\end{equation*}
which by Lemma \ref{orthog_lemma} is equal to the main term above.
\end{proof}

\begin{proof}[Proof of Lemma \ref{counting_lemma}]
Write $a= (a_i)_{i=1}^d$ for the coefficients of the quadratic system $\Psi$. Using the Fourier expansion of $F$ and multilinearity of $T$ one gets
\begin{multline*}
T(F\circ\Psi, F\circ\Psi, F\circ \Psi) = \sum_{\xi_1,\dots,\xi_6\in \Z^d} \hat{F}(\xi_1,\xi_2)\hat{F}(\xi_3,\xi_4)\hat{F}(\xi_5,\xi_6) T(f_1,f_2,f_3)
\end{multline*}
where $f_1(x) = e_p(\xi_1\cdot a\: x^2 + \xi_2\cdot a\: x), f_2(y) = e_p(\xi_3\cdot a\: y^2+\xi_4\cdot a\: y)$ and $f_3(z) = e_p(\xi_5\cdot a\: z^2+\xi_6\cdot a\: z)$. We proceed as in Proposition \ref{normbound_prop} to get that 
\begin{equation*}
T(f_1,f_2,f_3)= \sum_{\xi\in \fp} \hat{f}_1(-\xi) \hat{f}_2(-\xi)\hat{g}(\xi).
\end{equation*}
Now $\hat{f}_1(-\xi)\ll p^{-1/2}$ unless $\xi_1\cdot a = 0$ and $\xi_2\cdot a +\xi =0$, by \eqref{quadratic sum}. Similarly we have $\hat{f}_2(-\xi)\ll p^{-1/2}$ unless $\xi_3\cdot a=0$ and $\xi_4\cdot a+\xi = 0$, and from considering $f_3$ we have that $\hat{g}(\xi) \ll p^{-1/2}$ unless $\xi_5\cdot a - \xi = 0$ and $\xi_6\cdot a = 0$. If all of these fail at once we say that $(\xi_1,\dots,\xi_6,\xi)$ is \emph{exceptional}. For the sum over non-exceptional values we apply Proposition \ref{normbound_prop} to get that $T(f_1,f_2,f_3)\ll p^{-1/2}$, and then we bound the sum over $\xi_1,\dots,\xi_6$ by $M^3$.

For the exceptional values we have that $\xi = -\xi_2\cdot a= -\xi_4\cdot a=\xi_5\cdot a$ and $\xi_1\cdot a = \xi_3\cdot a=\xi_6\cdot a = 0$. Recalling \eqref{lambda} we may write this as
\begin{equation*}
T(F\circ\Psi,F\circ\Psi,F\circ\Psi) = \sum_{\substack{\xi_1,\xi_3,\xi_6\in\Lambda_\Psi\\\xi_2+\xi_5, \xi_4+\xi_5\in \Lambda_\Psi}} \hat{F}(\xi_1,\xi_2)\hat{F}(\xi_3,\xi_4)\hat{F}(\xi_5,\xi_6) + O(p^{-1/2}M^3).
\end{equation*}
By inserting the Fourier expansion of $F$ into the integral in \eqref{counting_eq} and using the orthogonality relation in Lemma \ref{orthog_lemma} we are done, in exactly the same way as in the proof of Lemma \ref{lem:baby_count}.
\end{proof}

\section{Ramsey lemma}\label{sec:ramsey}
The linear Ramsey problem addressed by the Ramsey lemma is to count triples $(t,u)$, $(t',u')$, $(t^{\prime\prime},u^{\prime\prime})$ such that $u+u' = t^{\prime\prime}$. In \cite{green-sanders} Green--Sanders's linear Ramsey problem appears quite similar to ours, namely finding $(t, u)$, $(t', u')$ and $(t^{\prime\prime}, u^{\prime\prime})$ such that $u^{\prime\prime} = t'-t$ and $u=u^\prime$. Although their problem doesn't directly imply ours, it turns out that their proof can be easily adapted to solve our problem, which is the approach we take here. We note that the basic scheme of their proof is inspired by \cite{cwalina}.

\begin{lem}\label{colouring lemma}
Let $(X,\nu_X)$ and $(Y,\nu_Y)$ be probability spaces, $A\subset X\times Y$, $(\nu_X\times\nu_Y)(A)=\alpha$ and $\eta\in (0,1]$ be a parameter. Then there is a measurable set $Y'\subset Y$, with $\nu_Y(Y')\ge \frac{1}{2}\alpha$, such that the set
\begin{equation*}
E = \Big\{ y\in Y: \nu_X({x\in X: (x,y)\in A})\le \frac{1}{2}\eta \alpha \Big\}
\end{equation*}
satisfies $\nu_Y(E\cap Y')\le \eta \nu_Y(Y')$.
\begin{proof}
Define
\begin{equation*}
\begin{split}
N_X(y) = \left\{ x\in X: (x,y)\in A \right\},\\
N_Y(x) = \left\{ y\in Y: (x,y)\in A \right\}.
\end{split}
\end{equation*}
Then
\begin{equation*}
\alpha = (\nu_X\times\nu_Y)(A) = \int \nu_X(N_X(y))\:d\nu_Y(y)
\end{equation*}
by Fubini's theorem. By the definition of $E$ we have
\begin{equation*}
\int 1_E (y) \nu_X(N_X(y))\:d\nu_Y(y) \le \frac{1}{2}\eta \alpha,
\end{equation*}
so that
\begin{equation*}
\int \Big( 1-\frac{1}{\eta}1_E(y) \Big) \nu_X(N_X(y))\:d\nu_Y(y) \ge \frac{1}{2}\alpha.
\end{equation*}
Another application of Fubini's theorem then gives
\begin{equation*}
\iint \left( 1-\frac{1}{\eta}1_E(y) \right) 1_{N_Y(x)}(y)\: d\nu_Y(y)d\nu_X(x) \ge \frac{1}{2}\alpha,
\end{equation*}
and in particular there is some $x\in X$ such that
\begin{equation}
\int \left( 1-\frac{1}{\eta}1_E(y) \right) 1_{N_Y(x)}(y) d\nu_Y(y) \ge \frac{1}{2}\alpha.
\label{nice stuff}
\end{equation}
For this $x$, set $Y' = N_Y(x)$. From \eqref{nice stuff} one gets that $\nu_Y(Y')\ge \frac{1}{2}\alpha$ and that $\nu_Y(E\cap Y')\le \eta \nu_Y(Y')$, as desired.
\end{proof}
\end{lem}

Throughout this section, let $G$ be a compact Abelian group with probability Haar measure $\mu$. For any measurable subset $T\subset G$ we write $\mu_T= \frac{1}{\mu(T)}\left.\mu\right|_T$. Define
\begin{equation}
\delta_T(A) = \int 1_A( t_1+t_2,t)\: d\mu_T^{\otimes 3}(t_1,t_2,t)
\label{delta_T}
\end{equation}
and
\begin{equation}
\Lambda_T(A) = \int 1_A(t_1+t_6,t_2)1_A(t_3+t_7,t_4)1_A(t_4+t_2,t_5) d\mu_T^{\otimes 7}(t_1,\dots,t_7).
\label{lambda_T}
\end{equation}

The main Ramsey lemma will follow from the following slightly stronger proposition. We need to allow for partial colourings and make $\rho(r)$ explicit in order to perform induction on $r$.

\begin{prop}\label{ramsey_prop}
Set $\epsilon_r=2^{-7r}(r!)^{-3}$. Suppose that $G$ is a compact Abelian group, $T\subset G$ is a measurable set, $E\subset (T+T)\times T$ is measurable and $c:(T+T)\times T\to [r]$ is a measurable partial colouring defined outside of $E$, where $\delta_T(E) \le \epsilon_r$. Then there is an $i\in[r]$ which satisfies $\Lambda_T(c^{-1}(i))\ge \epsilon_r^3$. 
\begin{proof}
We proceed by induction on $r$. For $r=1$ we have that $\Lambda_T(c^{-1}(1)) = \Lambda_T(((T+T)\times T)\setminus E) \ge 1-3\delta_T(E) \ge 1-3\eps_1>\eps_1^3$, and so the proposition holds in this case.

Assume that the statement holds for $r-1$ colours. Now $\epsilon_r < \frac{1}{2}$ so by the pigeon hole principle there is some $i$ with $\delta_T(A_i)\ge \frac{1}{2r}$, where we have defined 
\begin{equation*}
A_i = c^{-1}(i),\quad i=1,\dots,r.
\end{equation*}
Apply Lemma \ref{colouring lemma} with $X=T+T$, $Y=T$, $A=A_i$, $\eta=\frac{1}{6}\epsilon_{r-1}$, $\nu_Y=\mu_T$ and $\nu_X$ defined by
\[ \nu_X(C) = \int 1_C(t_1+t_2) \: d\mu_T^{\otimes 2}(t_1,t_2).\]
Note that $\nu_X\times \nu_Y=\delta_T$. The lemma gives us some $T'\subset T$ with $\mu_T(T')\ge \frac{1}{4r}$ and some $Z\subset T$ with a proportion $1-\eta$ of all of the elements in $T'$, which satisfies 
\begin{equation*}
\int 1_{A_i}(t_1+t_2,t)\:d\mu_T^{\otimes 2}(t_1,t_2)\ge \frac{1}{4r}\eta
\end{equation*}
for all $t$ in $Z$.

We now consider two cases. In the first case, assume that $\delta_{T'}(A_i)\ge \frac{1}{2}\epsilon_{r-1}$. Then restricting the integrals over $t_2,t_4$ and $t_5$ in \eqref{lambda_T} to $T'$ gives
\begin{multline*}
\Lambda_T(A_i) \ge \mu_T(T')^3 \int 1_{A_i}(t_2+t_4,t_5)\left(\int 1_{A_i}(t_1+t_6,t_2)\:d\mu_T^{\otimes 2}(t_1,t_6)\right)\\ \cdot\left( \int 1_{A_i}(t_3+t_7,t_4)\:d\mu_T^{\otimes 2}(t_3,t_7) \right)
\: d\mu_{T'}^{\otimes 3}(t_2,t_4,t_5)\\
\ge \mu_T(T')^3\frac{\eta^2}{2^4r^2} \int 1_{A_i}(t_2+t_4,t_5)1_Z(t_2)1_Z(t_4)\: d\mu_{T'}(t_2,t_4,t_5)\\
\ge \mu_T(T')^3\frac{\eta^2}{2^4r^2} \left( \int 1_{A_i}(t_2+t_4,t_5)\:d\mu_{T'}(t_2,t_4,t_5) - \frac{1}{3}\epsilon_{r-1} \right)\\
\ge \mu_T(T')^3\frac{\eta^2}{2^4r^2} \frac{1}{6}\epsilon_{r-1}\ge 2^{-13}3^{-3} r^{-5}\epsilon_{r-1}^3 = 2^83^{-3} r^4 \epsilon_r^3 > \epsilon_r^3,
\end{multline*}
and we are done in this case. If instead $\delta_{T'}(A_i)<\frac{1}{2}\epsilon_{r-1}$ we will use that
\begin{equation*}
\delta_{T'}(E) \le \mu_T(T')^{-3} \delta_T(E) \le 2^6 r^3\epsilon_r = \frac{1}{2}\epsilon_{r-1}.
\end{equation*}
Defining $E' = (E\cup A_i)\cap ( (T'+T')\times T')$ then gives $\delta_{T'}(E')\le \epsilon_{r-1}$. We now have a partial colouring of $(T'+T')\times T'$ outside $E'$ using $r-1$ colours, and the inductive hypothesis then gives some colour class $A_j$ with $\Lambda_{T'} ( A_j)\ge \epsilon_{r-1}^3$. Considering the corresponding colour class $A_j$ on $(T+T)\times T$ gives
\begin{equation*}
\Lambda_T(A_j) \ge \mu_T(T')^7 \Lambda_{T'}(A_j) \ge (4r)^{-7}\epsilon_{r-1}^3 = 2^7 r^2 \epsilon_r^3> \epsilon_r^3,
\end{equation*}
and so we are done.
\end{proof}
\end{prop}

\begin{proof}[Proof of Lemma \ref{ramsey_lemma}]
Put $f_i=\min(F_i,1)$, such that $0\le f_i \le 1$ and\\ $\sum_{i=1}^r f_i(g_1,g_2)\ge 1$ for all $(g_1,g_2)\in G\times G$. Define
\begin{equation}
A_i = \Big\{ (t,u)\in G\times G: f_i(t,u)\ge \frac{1}{r}, (t,u)\notin \bigcup_{j=1}^{i-1} A_j \Big\}.
\label{Ai}
\end{equation}
If $(t,u)\not\in\bigcup_i A_i$ then $\sum_i f_i(t,u)<1$, and so we must have $\bigcup_i A_i = G\times G$. Restricting to $A_i$ we get
\begin{multline*}
\int f_i(t,u) f_i(t',u')f_i(u+u',u^{\prime\prime})\:d\mu^{\otimes 5}(t,t',u,u',u^{\prime\prime})\\
\ge \frac{1}{r^3} \int 1_{A_i}(t,u) 1_{A_i}(t',u')1_{A_i}(u+u',u^{\prime\prime})\:d\mu^{\otimes 5}(t,t',u,u',u^{\prime\prime})
\end{multline*}
Introducing two dummy integrations and renaming the variables then gives
\begin{equation*}
\frac{1}{r^3} \int 1_{A_i}(t_1+t_6,t_2) 1_{A_i}(t_3+t_7,t_4)1_{A_i}(t_2+t_4,t_5)\:d\mu^{\otimes 7}(t_1,\dots,t_7).
\end{equation*}
Setting $T=G$ in Proposition \ref{ramsey_prop} we finally get that for some $i$ this is $\ge r^{-3}\epsilon_r^3$, which is the desired function $\rho(r)$.
\end{proof}

\section{Proof of main theorem in the general case}\label{sec:main2}
Before sketching the proof of Theorem \ref{thm:main2}, we note that one needs three different versions of the Ramsey lemma, depending on the exponents $\alpha,\beta,\gamma$ in \eqref{problem2}. The first version is needed when $\alpha=\beta=\gamma$, the second version is needed when exactly two of $\alpha,\beta$ and $\gamma$ are equal, and the third version is used when all of $\alpha,\beta$ and $\gamma$ are different. The necessity for different Ramsey lemmas arises because the conclusion of the counting lemma depends on the number of unique exponents, and this is of course what dictates the correct shape of the Ramsey lemma.

As mentioned in the introduction, the case where $\alpha=\beta=\gamma$ was already solved in \cite{csikvari-gyarmati-sarkozy} by using Schur's result on partition regularity of the equation $x+y=z$. Rather satisfyingly this is the same equation we end up having to address in the Ramsey lemma when using our methods.

At a first glance it may seem strange that the proof of Theorem \ref{thm:main} carries over to Theorem \ref{thm:main2} without any major modifications. In particular we note that the proof does not rely in any significant way on the fact that \eqref{problem} has a linear variable.

We begin by noting the changes that need to be made to the basic definitions in \S \ref{sec:count_sol} and \S \ref{sec:quad_trig}. This will essentially consist of redefining $T(f_1,f_2,f_3)$, $\Psi(x)$ and defining a new norm.

We consider the quantity
\begin{equation*}
T(f_1,f_2,f_3) = \frac{1}{p} \sum_{\substack{x,y,z\in\mb{F}_p\\ x^\alpha+y^\beta=z^\gamma}} f_1(x) f_2(y) f_3(z),
\end{equation*}
and define a norm
\begin{equation*}
\|f\|_{\alpha,\beta,\gamma} \coloneqq \sup\big\{|\mb{E}_x f(x) e_p(ax^\alpha+bx^\beta+cx^\gamma)|: a,b,c\in \mb{F}_p \big\}.
\end{equation*}
In this setting the analogue of Proposition \ref{normbound_prop} becomes the following.
\begin{prop}
If $\|f_1\|_2,\|f_2\|_2,\|f_3\|_2\le 1$ then
\begin{equation*}
|T(f_1,f_2,f_3)| \le\sqrt{\alpha\beta\gamma} \min_i (\|f_i\|_{\alpha,\beta,\gamma}).
\end{equation*}
\end{prop}
\begin{proof}[Sketch of proof]
Define $g_1(w) = \sum_{x^\alpha = w} f_1(x)$, $g_2(w) = \sum_{u^\beta=w}f_2(u)$ and $g_3(w) =\sum_{z^\gamma=w}f_3(z)$, and perform the same manipulations as in the proof of Proposition \ref{normbound_prop}.
\end{proof} 
We also have the analogue of Lemma \ref{l2_bound_lemma}, at the cost of a constant factor.
\begin{lem}
Let $f_1,f_2,f_3:\fp\to \C$. Then
\[ |T(f_1,f_2,f_3)| \le \sqrt{k} \|f_1\|_2 \|f_2\|_2 \|f_3\|_2 ,\]
where $k = \min\{\alpha,\beta,\gamma\}$.
\end{lem}
\begin{proof}[Sketch of proof]
Assume that $k=\alpha$ and define $g(w)=\sum_{x^\alpha=w} f_1(x)$. Doing as in the proof of Lemma \ref{l2_bound_lemma} and using the estimate $\|g\|_2 \le \sqrt{\alpha} \|f_1\|_2$ we are then done. The cases where $\alpha> k$ are treated similarly.
\end{proof}

In addition we will need a notion of a \emph{polynomial} system instead of just a quadratic system. In this section we let $\mb{G} = (\mb{R}/\mb{Z})\times (\mb{R}/\mb{Z})\times (\mb{R}/\mb{Z})$.
\begin{defn}
A polynomial system of dimension $d$ is a map $\Psi:\mb{F}_p\to \mb{G}^d$ of the form
\begin{equation*}
\Psi(x) = (a_i x^\alpha/p, a_i x^\beta/p, a_i x^\gamma/p)_{i=1}^d,
\end{equation*}
where $(a_i)_{i=1}^d \subset \mb{F}_p^d$.
\end{defn}
Furthermore we will need to define trigonometric polynomials on $\mb{G}^d$ instead of on $(\mb{R}/\mb{Z})^{2d}$ in the obvious way.

We now state the three key lemmas in this setting, and briefly sketch the proof of each.
\begin{lem}[Regularity lemma]
There are functions $p_3, M, D:\mb{Z}_{\ge 0}\times (0,1]\times \mb{N}\to \mb{R}_{\ge 0}$ such that the following holds. Let $c:\fp \to [r]$ be an $r$-colouring of $\fp$ and let $\eps>0$. Then there is a polynomial system $\Psi$ of dimension $d$, functions $F_1,\dots,F_r: \mb{G}^d\to \mb{R}_{\ge 0}$, and functions $g_1,\dots,g_r:\fp \to [-1,1]$, such that the following holds provided $p\ge p_3(r,\epsilon,K)$. 
\begin{enumerate}[(i)]
\item $\|F_i\|_\mathrm{trig} \le M(r,\epsilon,K)$ for all $i$;\\
\item $d\le D(r,\epsilon,K)$;\\
\item $\|F_i\circ \Psi-g_i\|_2\le \epsilon$ for all $i$;\\
\item $\|1_{c^{-1}(i)}-g_i\| \le \epsilon$ for all $i$;\\
\item $\sum_i F_i\circ \Psi\ge 1$ pointwise.\\
\end{enumerate}
Here $K=\max\{\alpha,\beta,\gamma\}$.
\end{lem}
\begin{proof}[Sketch of proof]
    Lemma \ref{lem:regularity_lemma} is proved via the three intermediate lemmas \ref{u3bad_projbad_lemma}, \ref{koopman_von_neumann} and \ref{trig_lem}. The analogue of each of these in the current setting goes through with nearly identical proofs, by modifying a few constants. In particular we note that in the proof of Lemma \ref{trig_lem} we will need the following fact instead of \eqref{interval_bd}. Let $J$ be an interval in $\R/\Z$ not containing $0$, and let $\delta\in \N$. Then
\[ \#\{ x\in \fp: ax^\delta/p \in J\} \le \delta p |J| + C\]
for any $a\in \fp$ and an absolute constant $C$. This holds because $ax^\delta$ takes each value in $\fp$ at most $\delta$ times.

Combining the three is also done in exactly the same way as before.
\end{proof}

The conclusion of the counting lemma looks slightly different depending on how many of $\alpha,\beta$ and $\gamma$ are equal. In general we may write the integrand in the counting lemma as
\[ F(t_\alpha,t_\beta,t_\gamma) F(u_\alpha,u_\beta,u_\gamma)F(v_\alpha,v_\beta,v_\gamma)1_{t_\alpha+u_\beta=v_\gamma}.\]
If say $\alpha=\beta=\gamma$ this should be interpreted as
\[\tilde{F}(t)\tilde{F}(u)\tilde{F}(t+u),\]
where we have defined $\tilde{F}(w)=F(w,w,w)$. Similarly, if $\alpha=\beta\ne \gamma$ we recover the integrand in Lemma \ref{counting_lemma} with $\ol{F}$ in place of $F$, where we define $\ol{F}(u,v) = F(v,v,u)$. Below we state and sketch the proof of the counting lemma for the case where none of $\alpha,\beta,\gamma$ are equal, and leave it to the reader to examine the other cases.

\begin{lem}[Counting lemma]
Let $\Psi$ be a $d$-dimensional polynomial system with distinct exponents $\alpha,\beta,\gamma$, and let $F:\mb{G}^d\to \mb{C}$ be a trigonometric polynomial. Then
\begin{multline*}
T(F\circ \Psi, F\circ\Psi, F\circ\Psi) \\= \int F(t_1,u_1,v_1) F(t_2,u_2,v_2)F(t_3,u_3,t_1+u_2)\: d\mu_{G_\Psi}^{\otimes 8}(t_1,\dots,u_3)\\
+ O_K (p^{-1/2^{K-1}}M^3),
\end{multline*}
where $K = \max\{\alpha,\beta,\gamma\}$ and $M = \|F\|_\mathrm{trig}$.
\end{lem}
\begin{proof}[Sketch of proof]
We follow the steps of the proof of Lemma \ref{counting_lemma}, with the standard bound
\begin{equation*}
\mb{E}_x e_p(g(x))\ll_K p^{-1/2^{K-1}}
\end{equation*}
for any non-zero polynomial $g$ of degree at most $K$, in place of \eqref{quadratic sum}. This bound can be obtained by the van der Corput estimate. This will give
\begin{multline*}
T(F\circ \Psi, F\circ\Psi, F\circ\Psi)\\ = C_{\alpha,\beta,\gamma}\sum_{\substack{\xi_2,\xi_3,\xi_4,\xi_6,\xi_7,\xi_8\in\Lambda_\Psi\\ \xi_1+\xi_9,\xi_5+\xi_9\in\Lambda_\Psi}} \hat{F}(\xi_1,\xi_2,\xi_3)\hat{F}(\xi_4,\xi_5,\xi_6)\hat{F}(\xi_7,\xi_8,\xi_9)\\
+ O_K(p^{-1/2^{K-1}}M^3),
\end{multline*}
where
\begin{multline*}
C_{\alpha,\beta,\gamma} = \frac{1}{p^2} \sum_{x^\alpha+y^\beta=z^\gamma} 1 = \mb{E}_{u,v} f_\alpha(u)f_\beta(v)f_\gamma(u+v) = \sum_{\xi} \hat{f}_\alpha(\xi)\hat{f}_\beta(\xi)\hat{f}_\gamma(-\xi)\\
= \hat{f}_\alpha(0)\hat{f}_\beta(0)\hat{f}_\gamma(0) + O_K(p^{-1/2^{K-1}}),
\end{multline*}
where $f_\delta(u) = \sum_{x^\delta=u} 1$ for $\delta\in\mb{N}$, and we have bounded the terms with $\xi\ne 0$ as usual by noting $|\hat{f}_\alpha(\xi)|\ll p^{-1/2^{K-1}}$ and $\|f_\beta\|_2\le \sqrt{\beta},\|f_\gamma\|_2\le \sqrt{\gamma}$. Now
\begin{equation*}
\hat{f}_\delta(0) = \mb{E}_u f_\delta (u) = \mb{E}_x 1 = 1
\end{equation*}
for $\delta\in\mb{N}$, and so $C_{\alpha,\beta,\gamma} = 1 + O_K(p^{-1/2^{K-1}})$, which completes the proof.
\end{proof}
Note that the constant $C_{1,1,2}$ was implicitly evaluated to $1$ in the proof of Lemma \ref{counting_lemma}, which is why it wasn't necessary to introduce this quantity there. 

We also note that if some of $\alpha,\beta,\gamma$ are equal, the values of $\xi_1,\dots,\xi_9$ which contribute to the main term will be different. This is because for a sum of the form $ \E_x e_p(ax^\alpha+bx^\beta+c x^\gamma)$ to be small it is not enough that $(a,b,c)\ne (0,0,0)$ in this case.

Finally, because of the different forms of the counting lemma, the Ramsey lemma will have to look slightly different depending on which case we are in. In the second case the needed Ramsey lemma is in fact Lemma \ref{lem:regularity_lemma}. We here therefore prove the first and the third case.
\begin{lem}[Ramsey lemma, case 1]
There is a positive function $\rho:\mb{Z}_{\ge 0} \to (0,1]$ with the following property. Assume that $F_1,\dots,F_r: G_\Psi \to \mb{R}_{\ge 0}$ are continuous functions that satisfy $\sum_{i=1}^r F_i(x)\ge 1$ for all $x\in G_\Psi$. Then there is an $i\in [r]$ such that
\begin{equation*}
\int F_i(t_1) F_i(t_2)F_i(t_1+t_2)\:d\mu^{\otimes 2}_{G_\Psi}(t_1,t_2) \ge \rho(r).
\end{equation*}
\end{lem}
\begin{proof}
We could prove this by using the same techniques as in the proof of Lemma \ref{ramsey_lemma}, but in this case our result actually just follows from a quantitative version of Schur's theorem. Note that if $a \equiv 0\Mod p$ then $G_\Psi = \{0\}$, and if $a_i\nequiv 0\Mod p$ for some $i$ we have $G_\Psi = \{ax/p: x\in\fp\}\simeq \fp$, by the definition of $G_\Psi$. If $G_\Psi = \{0\}$ the statement we are trying to prove is trivial, as $F(0)\ge 1/r$ for some $i$, and for this $i$ the above integral is $\ge 1/r^3$.

In the non-trivial case, define sets $A_i$ as in the proof of Lemma \ref{ramsey_lemma} to get an $r$ colouring of $G_\Psi$. We then invoke \cite[Theorem 1]{frankl}, which states that given an $r$-colouring of $\N$ there are $\ge \nu(r) N^{2}$ monochromatic solutions to $x+y=z$ with $x,y,z\le N$. Set $p=N$, then certainly we also have $\ge \nu(r) p^2$ monochromatic solutions to $x+y=z$ with $x,y,z\in \fp$, which completes the proof.
\end{proof}

\begin{lem}[Ramsey lemma, case 3]\label{lem:ramsey2}
There is a positive function $\rho:\mb{Z}_{\ge 0} \to (0,1]$ with the following property. Suppose that $G$ is a compact Abelian group with Haar probability measure $\mu$. Assume also that $F_1,\dots,F_r: G\times G\times G\to \mb{R}_{\ge 0}$ are continuous functions that satisfy $\sum_{i=1}^r F_i(g_1,g_2,g_3)\ge 1$ for all $(g_1,g_2,g_3)\in G\times G\times G$. Then there is an $i\in [r]$ such that
\begin{equation*}
\int F_i(t_1,u_1,v_1) F_i(t_2,u_2,v_2)F_i(t_3,u_3,t_1+u_2)\:d\mu^{\otimes 8}(t_1,\dots,u_3) \ge \rho(r).
\end{equation*}
\end{lem}

\begin{proof}[Sketch of proof of Lemma \ref{lem:ramsey2}]
The linear Ramsey problem we are addressing is slightly different from the one in Lemma \ref{ramsey_lemma}, but a similar proof goes through. Introducing two dummy integrations we can view this as a problem on $T\times T\times (T+T)$, and the heart of the argument is to prove an analogue of Proposition \ref{ramsey_prop} in this setting, where now
\[ \Lambda_T(A) = \int 1_A(t_1,u_1,v_1+w_1) 1_A(t_2,u_2,v_2+w_2)1_A(t_3,u_3,t_1+u_2)\:d\mu^{\otimes 10}_T(\dots)\]
and
\[ \delta_T(A) = \int 1_A(t,u,v+w)\:d\mu_T^{\otimes 4}(t,u,v,w).\]
We give a rough sketch of this proof without specifying any parameters or constants.

As in the proof of Proposition \ref{ramsey_prop}, let $A_i$ be the largest colour class on $T$. Invoke Lemma \ref{colouring lemma} with $X= T\times (T+T)$ and $Y= T$ to get a set $T'\subset T$ such that
\[\int 1_{A_i}(t,u,v+w)\: d\mu_T^{\otimes 3}(u,v,w)\]
is large for all but a fraction $\eta$ of $t\in T'$. Restrict $t_1,t_2,u_2,v_2,w_2,t_3,u_3$ to this set. This then gives
\[\Lambda_T(A_i) \ge C(r) \int 1_{A_i}(t_2,u_2,v_2+w_2)1_{A_i}(t_3,u_3,t_1+u_2)\:d\mu_{T'}^{\otimes 7}(\dots)\]
for some quantity $C(r)$ depending only on $r$. Now, if $\delta_{T'}(A_i)$ is small we pass to a partial colouring of $T'$, removing the colour $i$, and use the induction hypothesis. If $\delta_{T'}(A_i)$ is large, we apply Lemma \ref{colouring lemma} a second time to get a set $T''\subset T'$ such that 
\[ \int 1_{A_i}(t,u,v+w)\:d\mu_{T'}^{\otimes 3}(t,v,w)\]
is large for all but a fraction $\eta$ of $u\in T''$. Restricting $u_2,t_3,u_3,t_1$ to this set then gives
\[\Lambda_T(A_i) \ge C'(r) \int 1_{A_i}(t_3,u_3,t_1+u_2)\:d\mu_{T''}^{\otimes 4}(t_3,u_3,t_1,u_2)=C'(r)\delta_{T''}(A_i).\]
To complete the proof, note that either $\delta_{T''}(A_i)$ is large, in which case we are done, or $\delta_{T''}(A_i)$ is small and we can pass to an $r-1$-colouring of $T''$ and use the induction hypothesis.
\end{proof}

Finally the three lemmas are combined in almost exactly the same way as in \S \ref{sec:main}.

\bibliographystyle{plain}
\bibliography{literature}

\end{document}